\documentclass[a4paper,reqno,10pt]{amsart}

\usepackage{hyperref}
\usepackage{a4wide}

\usepackage{amssymb}
\usepackage{amstext}
\usepackage{amsmath}
\usepackage{amscd}
\usepackage{amsthm}
\usepackage{amsfonts}

\usepackage{graphicx}
\usepackage{latexsym}
\usepackage{mathrsfs}

\usepackage{enumitem}
\setlist[enumerate,1]{label={\upshape(\arabic*)}}
\setlist[enumerate,2]{label={\upshape(\alph*)}}

\usepackage{tikz}
\usetikzlibrary{cd,arrows,matrix,backgrounds,positioning,calc,decorations.markings,decorations.pathmorphing,decorations.pathreplacing}
\tikzset{black/.style={circle,fill=black,inner sep=3pt,outer sep=3pt},
         white/.style={circle,fill=white,draw=black,inner sep=3pt,outer sep=3pt},
}

\usepackage{array}
\newcolumntype{C}{>{$}c<{$}}

\usepackage{makecell}

\newtheorem{theorem}{Theorem}[section]
\newtheorem{theoremi}{Theorem}

\newtheorem{corollary}[theorem]{Corollary}
\newtheorem{lemma}[theorem]{Lemma}
\newtheorem*{lemma*}{Lemma}
\newtheorem*{theorem*}{Theorem}
\newtheorem{proposition}[theorem]{Proposition}
\newtheorem{definition-proposition}[theorem]{Definition-Proposition}

\theoremstyle{definition}
\newtheorem{definition}[theorem]{Definition}

\newtheorem{remark}[theorem]{Remark}

\newtheorem*{ack}{Acknowledgement}
\newtheorem*{conv}{Conventions and notation}
\newtheorem*{org}{Organization}

\newcommand{\qedb}{\hfill\blacksquare}

\newcommand{\la}{\langle}
\newcommand{\ra}{\rangle}

\renewcommand{\AA}{\mathcal{A}}

\newcommand{\CC}{\mathcal{C}}

\newcommand{\PP}{\mathcal{P}}

\newcommand{\TT}{\mathcal{T}}
\newcommand{\TTT}{\mathsf{T}}
\newcommand{\UU}{\mathcal{U}}
\newcommand{\WW}{\mathcal{W}}
\newcommand{\WWW}{\mathsf{W}}
\newcommand{\XX}{\mathcal{X}}

\newcommand{\Ext}{\operatorname{Ext}\nolimits}

\newcommand{\Hom}{\operatorname{Hom}\nolimits}

\newcommand{\End}{\operatorname{End}\nolimits}

\newcommand{\RHom}{\mathbf{R}\strut\kern-.2em\operatorname{Hom}\nolimits}

\newcommand{\Image}{\operatorname{Im}\nolimits}
\newcommand{\Kernel}{\operatorname{Ker}\nolimits}
\newcommand{\Cokernel}{\operatorname{Coker}\nolimits}

\newcommand{\coker}{\Cokernel}
\newcommand{\im}{\Image}
\renewcommand{\ker}{\Kernel}

\newcommand{\ov}{\overline}

\DeclareMathOperator{\moduleCategory}{\mathsf{mod}} \renewcommand{\mod}{\moduleCategory}

\DeclareMathOperator{\ice}{\mathsf{ice}}
\DeclareMathOperator{\icep}{\mathsf{ice_p}}

\DeclareMathOperator{\ftors}{\mathsf{f-tors}}
\DeclareMathOperator{\tors}{\mathsf{tors}}

\DeclareMathOperator{\rigid}{\mathsf{rigid}}
\DeclareMathOperator{\rigidz}{\mathsf{rigid_0}}

\DeclareMathOperator{\stilt}{\mathsf{stilt}}

\DeclareMathOperator{\wide}{\mathsf{wide}}

\DeclareMathOperator{\fwide}{\mathsf{f-wide}}

\DeclareMathOperator{\Sub}{\mathsf{Sub}}

\DeclareMathOperator{\Fac}{\mathsf{Fac}}

\DeclareMathOperator{\add}{\mathsf{add}}

\DeclareMathOperator{\ccok}{\mathsf{cok}}

\newcommand{\iso}{\cong}

\newcommand{\defl}{\twoheadrightarrow}

\newcommand{\equi}{\simeq}

\numberwithin{equation}{section}

\begin{document}
\title[Rigid modules and ICE-closed subcategories in quiver representations]{Rigid modules and ICE-closed subcategories \\in quiver representations}

\author[H. Enomoto]{Haruhisa Enomoto}
\address{Graduate School of Mathematics, Nagoya University, Chikusa-ku, Nagoya. 464-8602, Japan}
\email{m16009t@math.nagoya-u.ac.jp}
\begin{abstract}
  We introduce image-cokernel-extension-closed (ICE-closed) subcategories of module categories. This class  unifies both torsion classes and wide subcategories. We show that ICE-closed subcategories over the path algebra of Dynkin type are in bijection with basic rigid modules, that ICE-closed subcategories are precisely torsion classes in some wide subcategories, and that the number does not depend on the orientation of the quiver. We give an explicit formula of this number for each Dynkin type, and in particular, it is equal to the large Schr\"oder number for type A case.
\end{abstract}

\maketitle

\tableofcontents

\section{Introduction}
Let $\Lambda$ be an artin algebra and $\mod\Lambda$ the category of finitely generated right $\Lambda$-modules.
There are several kinds of subcategories of $\mod\Lambda$ which have been investigated in the representation theory of algebras, e.g. \emph{wide} subcategories, \emph{torsion classes}, \emph{torsion-free classes}, and so on. These subcategories are defined by the property that they are closed under certain operations: e.g.  taking kernels, cokernels, images, extensions, submodules, or quotients.

In this paper, we propose a new class of subcategories of $\mod\Lambda$, \emph{ICE-closed} subcategories, which is a subcategory closed under Images, Cokernels and Extensions. Typical examples of ICE-closed subcategories are torsion classes and wide subcategories, but there are more than them.

Recently, there are lots of studies on the classification of \emph{nice} subcategories in terms of \emph{nice} modules. One of the most prominent results is the \emph{$\tau$-tilting theory} established in \cite{AIR}, which gives a bijection between functorially finite torsion classes and certain class of modules called \emph{support $\tau$-tilting modules}. Actually, it is a generalization of the \emph{Ingalls-Thomas bijection} \cite{IT}, which classifies functorially finite torsion classes over hereditary algebras by \emph{support tilting} modules.

The aim of this paper is to provide a similar classification of ICE-closed subcategories over hereditary algebras. More precisely, we show that such subcategories are in bijection with \emph{rigid modules},  modules without self-extensions.
The main result is summarized as follows:
\begin{theoremi}[= Theorem \ref{ice:thm:main}, Corollary \ref{cor:exc}]
  Let $Q$ be a Dynkin quiver. Then there is a bijection between the following two sets:
  \begin{enumerate}
    \item The set of ICE-closed subcategories of $\mod kQ$.
    \item The set of isomorphism classes of basic rigid $kQ$-modules.
  \end{enumerate}
  The bijections are given as follows.
  \begin{itemize}
    \item For an ICE-closed subcategory $\CC$, the corresponding rigid $kQ$-module is the basic $\Ext$-progenerator $P(\CC)$ of $\CC$.
    \item For a rigid $kQ$-module $U$, the corresponding ICE-closed subcategory is given by the category $\ccok U$ consisting of cokernels of maps in $\add U$.
  \end{itemize}
   Moreover, a subcategory of $\mod\Lambda$ is ICE-closed if and only if it is a torsion class in some wide subcategory of $\mod kQ$.
\end{theoremi}
Actually, this theorem holds for any representation-finite hereditary artin algebras, and can be generalized to representation-infinite case by restricting the class of ICE-closed subcategories, see Theorem \ref{ice:thm:main} for the precise statement.

In the appendix, we will count the number of ICE-closed subcategories by using several results in other papers. The result is summarized as follows.
\begin{theoremi}
  Let $Q$ be a Dynkin quiver. Then the number of ICE-closed subcategories only depends on the underlying Dynkin graph, not on the choice of an orientation (Theorem \ref{ice:thm:mutinv}). Moreover, we have an explicit formula of this number for each Dynkin type, and if $Q$ is of type $A_n$, then it is equal to the $n$-th large Schr\"oder number (Corollay \ref{ice:cor:count}).
\end{theoremi}

We expect that there is a hidden theory which generalizes this paper to non-hereditary case, as \cite{AIR} generalizes \cite{IT}.

\begin{org}
  This paper is organized as follows.
  In Section \ref{ice:sec:2}, we give basic definitions and state a main result Theorem \ref{ice:thm:main}.
  In Section \ref{ice:sec:3}, we give a proof of Theorem \ref{ice:thm:main}.
  In Section \ref{ice:sec:4}, we consider the relation between ICE-closed subcategories, torsion classes and wide subcategories in detail via rigid modules.
  In Section \ref{sec:5}, we give another proof of Theorem \ref{ice:thm:main} using exceptional sequences, and discuss its consequences.
  In the appendix, we count the number of ICE-closed subcategories for each Dynkin type.
\end{org}

\begin{conv}
  Throughout this paper, \emph{all subcategories are assumed to be full and closed under isomorphisms, direct sums and direct summands}. An \emph{artin $R$-algebra} is an $R$-algebra over a commutative artinian ring $R$ which is finitely generated as an $R$-module. We often omit the base ring $R$, and simply call it an \emph{artin algebra}.

  For an artin algebra $\Lambda$, we denote by $\mod\Lambda$ the category of finitely generated right $\Lambda$-modules. All modules are finitely generated right modules. For a collection $\CC$ of $\Lambda$-modules, we denote by $\add \CC$ the subcategory of $\mod\Lambda$ consisting of direct summands of finite direct sums of objects in $\CC$.
  A module $M$ is called \emph{basic} if there is a decomposition $M = \bigoplus_{i=1}^n M_i$ such that each $M_i$ is indecomposable and pairwise non-isomorphic. For a module $M$, we denote by $|M|$ the number of non-isomorphic indecomposable direct summands of $M$.
\end{conv}

\section{Basic definitions and the main result}\label{ice:sec:2}
In this section, we give basic definitions and introduce some notation, and state our main result.
First of all, recall that a module $M \in \mod\Lambda$ over an artin algebra $\Lambda$ is \emph{rigid} if $\Ext_\Lambda^1(M,M) = 0$ holds.
Then we define several conditions on the subcategory of $\mod\Lambda$.
\begin{definition}
  Let $\Lambda$ be an artin algebra and $\CC$ a subcategory of $\mod\Lambda$.
  \begin{enumerate}
    \item $\CC$ is \emph{closed under extensions} if for every short exact sequence in $\mod\Lambda$
    \[
    \begin{tikzcd}
      0 \rar & L \rar & M \rar & N \rar & 0,
    \end{tikzcd}
    \]
    we have that $L,N \in \CC$ implies $M \in \CC$
    \item $\CC$ is \emph{closed under quotients (resp. submodules)} if for every short exact sequence in $\mod\Lambda$
    \[
    \begin{tikzcd}
      0 \rar & L \rar & M \rar & N \rar & 0,
    \end{tikzcd}
    \]
    we have that $M \in \CC$ implies $N \in \CC$ (resp. $L \in \CC$).
    \item $\CC$ is \emph{closed under cokernels (resp. images)} if for every map $f \colon M \to N$ with $M,N \in \CC$, we have $\coker f \in \CC$ (resp. $\im f \in \CC$).
    \item $\CC$ is a \emph{torsion class} if $\CC$ is closed under quotients and extensions.
    \item $\CC$ is a \emph{wide subcategory} if $\CC$ is closed under kernels, cokernels and extensions.
    \item $\CC$ is \emph{image-cokernel-extension-closed}, abbreviated by \emph{ICE-closed}, if $\CC$ is closed under images, cokernels and extensions.
    \item $\CC$ is \emph{cokernel-extension-closed}, abbreviated by \emph{CE-closed}, if $\CC$ is closed under cokernels and extensions.
  \end{enumerate}
\end{definition}
Then clearly all torsion classes and wide subcategories are (I)CE-closed, thus ICE-closed subcategories can be seen as a generalization of these two classes.

To an extension-closed subcategory of $\mod\Lambda$, we can associate a rigid module by taking the \emph{$\Ext$-progenerator}.
\begin{definition}
  Let $\Lambda$ be an artin algebra and $\CC$ an extension-closed subcategory of $\mod\Lambda$.
  \begin{enumerate}
    \item An object $X \in \CC$ is \emph{$\Ext$-projective in $\CC$} if $\Ext_\Lambda^1(X,\CC) = 0$ holds. We denote by $\PP(\CC)$ the subcategory of $\CC$ consisting of all the $\Ext$-projective objects in $\CC$.
    \item $\CC$ has \emph{enough $\Ext$-projectives} if for every object $X \in \CC$, there is a short exact sequence
    \[
    \begin{tikzcd}
      0 \rar & Y \rar & P \rar & X \rar & 0
    \end{tikzcd}
    \]
    with $P \in \PP(\CC)$ and $Y \in \CC$.
    \item An object $P \in \CC$ is an \emph{$\Ext$-progenerator} if $\add P = \PP(\CC)$ and $\CC$ has enough $\Ext$-projectives.
  \end{enumerate}
\end{definition}

Now we are ready to state our main result. Throughout this paper, we will use the following notations for an artin algebra $\Lambda$.
\begin{itemize}
  \item $\rigid \Lambda$ denotes the set of isomorphism classes of basic rigid $\Lambda$-modules.
  \item $\ice\Lambda$ denotes the set of ICE-closed subcategories of $\mod\Lambda$.
  \item $\icep \Lambda$ denotes the set of ICE-closed subcategories of $\mod\Lambda$ with enough $\Ext$-projectives.
  \item For an extension-closed subcategory $\CC$ of $\mod\Lambda$ with an $\Ext$-progenerator, we denote by $P(\CC)$ the unique basic $\Ext$-progenerator of $\CC$.
  \item For a $\Lambda$-module $U$, we denote by $\ccok U$ the subcategory of $\mod \Lambda$ consisting of cokernels of maps in $\add U$.
  \item For a collection $\UU$ of $\Lambda$-modules, we denote by $\Fac \UU$ (resp. $\Sub \UU$) the subcategory of $\mod \Lambda$ consisting of quotients (resp. submodules) of objects in $\add \UU$.
\end{itemize}

\begin{theorem}\label{ice:thm:main}
  Let $\Lambda$ be a hereditary artin algebra. Then we have the following bijections
  \[
  \begin{tikzcd}
    \rigid \Lambda \rar[shift left, "\ccok"] & \icep\Lambda. \lar[shift left, "P"]
  \end{tikzcd}
  \]
  Moreover, if $\Lambda$ is representation-finite, then every CE-closed subcategory is automatically ICE-closed, and $\icep\Lambda= \ice\Lambda$ holds.
\end{theorem}
This bijection extends a bijection between support ($\tau$-)tilting modules and functorially finite torsion classes given in \cite{IT} or \cite{AIR} in the following sense. Let $\Lambda$ be a hereditary artin algebra. It is known that a torsion class is functorially finite if and only if it has enough $\Ext$-projectives. Then the following diagram commutes, and the horizontal maps are bijective.
\[
\begin{tikzcd}
  \rigid \Lambda \rar[shift left, "\ccok"] & \icep\Lambda \lar[shift left, "P"] \\
  \stilt \Lambda \uar[hookrightarrow] \rar[shift left, "\Fac"] & \ftors\Lambda \lar[shift left, "P"] \uar[hookrightarrow]
\end{tikzcd}
\]
Here the bottom bijections were those given in \cite{IT} or \cite{AIR}. See Proposition \ref{ice:prop:tiltcompati} for the detail.

\begin{remark}
  ICE-closed subcategories generalize both torsion classes and wide subcategories. Another generalization of these two classes was introduced in \cite{enomono}, \emph{right Schur subcategories} (actually the dual was studied in the paper). Every ICE-closed subcategory is a right Schur, but the converse does nod hold in general. If $\Lambda$ is Nakayama, then these coincides by \cite[Theorem 6.1]{enomono}.
  In \cite{enomono}, we classify right Schur subcategories in any length abelian category by using \emph{simple objects} in them. This is in contrast with our use of $\Ext$-projectives.
\end{remark}

\section{Proof of the main theorem}\label{ice:sec:3}
In this section, we give a proof of Theorem \ref{ice:thm:main}. \emph{Throughout this section, we denote by $\Lambda$ a hereditary artin algebra}.

First we give a map $P(-)\colon \icep\Lambda \to \rigid\Lambda$.
\begin{proposition}\label{ice:prop:icerigid}
  Let $\CC$ be an CE-closed subcategory of $\mod\Lambda$. Then the following hold.
  \begin{enumerate}
    \item Every $\Ext$-projective object in $\CC$ is rigid.
    \item There are only finitely many indecomposable $\Ext$-projective objects in $\CC$ up to isomorphism.
    \item If $\CC$ has enough $\Ext$-projectives, then it has an $\Ext$-progenerator $P(\CC)$, and $\CC = \ccok P(\CC)$ holds.
  \end{enumerate}
\end{proposition}
\begin{proof}
  (1)
  Clear from definition.

  (2)
  We claim that there are at most $|\Lambda|$ indecomposable $\Ext$-projectives in $\CC$. Let $M_1,\dots,M_m$ be pairwise non-isomorphic $\Ext$-projectives in $\CC$. Then clearly $M:= M_1 \oplus \cdots \oplus M_m$ is basic rigid, or partial tilting since $\Lambda$ is hereditary. Then by taking the Bongartz  completion, there is a $\Lambda$-module $N$ such that $M \oplus N$ is a basic tilting $\Lambda$-module (see \cite[Lemma VI.2.4]{ASS} for the detail). It follows that $m = |M| \leq |M \oplus N| = |\Lambda|$.

  (3)
  Since $\CC$ has enough $\Ext$-projectives, (2) implies that $\CC$ has an $\Ext$-progenerator $P(\CC)$.
  Since $\CC$ is closed under cokernels and $P(\CC) \in \CC$, clearly $\CC\supset \ccok P(\CC)$ holds. Conversely, we have $\CC \subset \ccok P(\CC)$ since $\CC$ has enough $\Ext$-projectives.
\end{proof}

The following lemma is essential in our proof. This says that \emph{$\add U$ is closed under images if $\Lambda$ is herediatry and $U$ is rigid}.
\begin{lemma}\label{ice:lem:image}
  Let $\Lambda$ be a hereditary artin algebra and $U$ a rigid $\Lambda$-module. Then $\Fac U \cap \Sub U = \add U$ holds.
\end{lemma}
\begin{proof}
  Clearly we have $\add U \subset \Fac U \cap \Sub U$. Conversely, let $X \in \Fac U \cap \Sub U$. Take a left $(\add U)$-approximation $\varphi \colon X \to U^X$ with $U^X \in \add U$, which is an injection by $X \in \Sub U$.
  Then we have the following commutative exact diagram in $\mod \Lambda$:
  \[
  \begin{tikzcd}
    &U_0 \dar[twoheadrightarrow] \\
    0 \rar & X \rar["\varphi"]& U^X \rar & C \rar & 0 \\
  \end{tikzcd}
  \]
  By applying $\Hom(-,U)$, we obtain an exact sequence
  \[
  \begin{tikzcd}
    \Hom_\Lambda(U^X,U) \rar["{(-)\circ\varphi}"] &  \Hom_\Lambda(X,U) \rar & \Ext_\Lambda^1(C,U) \rar &  \Ext_\Lambda^1(U^X,U).
  \end{tikzcd}
  \]
  Since $\varphi$ is a left $(\add U)$-approximation, $(-)\circ\varphi$ is a surjection.
  In addition, $\Ext_\Lambda^1(U^X,U)$ vanishes since $U$ is rigid, hence $\Ext_\Lambda^1(C,U) = 0$. On the other hand, since $\Lambda$ is hereditary, we have an exact sequence
  \[
  \begin{tikzcd}
    \Ext_\Lambda^1(C,U_0) \rar & \Ext_\Lambda^1(C,X) \rar & 0.
  \end{tikzcd}
  \]
  Since we have $\Ext_\Lambda^1(C,U_0) = 0$, we obtain $\Ext_\Lambda^1(C,X) = 0$. It follows that the short exact sequence $0 \to X \to U^X \to C \to 0$ splits, which implies $X \in \add U$.
\end{proof}

Our next aim is to show that $\ccok U$ is ICE-closed if $U$ is rigid. We will make use of the subcategory $\XX_U$ associated to $U$, which was introduced by Auslander-Reiten \cite{applications}.
\begin{definition}
  Let $\Lambda$ be an artin algebra and $U$ a $\Lambda$-module with $\Ext_\Lambda^{>0}(U,U) = 0$. Then we denote by $\XX_U$ a subcategory of $\mod\Lambda$ consisting of modules $X$ such that there is an exact sequence
  \[
  \cdots \xrightarrow{f_2} U_1 \xrightarrow{f_1} U_0 \xrightarrow{f_0} X \to 0
  \]
  with $\Ext_\Lambda^{>0}(U,\im f_i) = 0$ for all $i \geq 0$.
\end{definition}
We borrow the following lemma from \cite{applications}.
\begin{lemma}[{\cite[Proposition 5.1]{applications}}]\label{ice:lem:ar}
  Let $\Lambda$ be an artin algebra and $U$ a $\Lambda$-module with $\Ext_\Lambda^{>0}(U,U)\allowbreak = 0$. Then $\XX_U$ is closed under extensions, has an $\Ext$-progenerator $U$, and is closed under mono-cokernels, that is, for every short exact sequence
  \[
  \begin{tikzcd}
    0 \rar & L \rar & M \rar & N \rar & 0
  \end{tikzcd}
  \]
  in $\mod\Lambda$, if $L$ and $M$ belong to $\XX_U$, then so does $N$.
\end{lemma}
The following is basic properties of $\XX_U$ in our hereditary setting. In particular, we have $\XX_U = \ccok U$ for a rigid module $U$ over a hereditary algebra $\Lambda$.
\begin{proposition}\label{ice:prop:rigidice}
  Let $\Lambda$ be a hereditary artin algebra and $U$ a rigid $\Lambda$-module.
  \begin{enumerate}
    \item The following are equivalent for $X \in \mod\Lambda$.
      \begin{enumerate}
        \item $X$ belongs to $\XX_U$.
        \item $X$ belongs to $\ccok U$.
        \item There is an short exact sequence of the following form with $U_1$ and $U_0$ in $\add U$.
        \[
        \begin{tikzcd}
          0 \rar & U_1 \rar & U_0 \rar & X \rar & 0
        \end{tikzcd}
        \]
      \end{enumerate}
     \item $\ccok U$ is an ICE-closed subcategory of $\mod\Lambda$.
     \item $\ccok U$ has an $\Ext$-progenerator $U$.
  \end{enumerate}
\end{proposition}
\begin{proof}
  (1)
  Clearly (a) implies (b) by $\XX_U \subset \ccok U$. Also the implication (b) $\Rightarrow$ (c) follows immediately from Lemma \ref{ice:lem:image} since $\Lambda$ is hereditary and $U$ is rigid, and Lemma \ref{ice:lem:ar} shows (c) $\Rightarrow$ (a).

  (2)
  By Lemma \ref{ice:lem:ar}, we only have to show that $\XX_U = \ccok U$ is closed under images, because this will immediately imply that $\XX_U$ is closed under cokernels since $\XX_U$ is closed under mono-cokernels.

  Take any $\varphi \colon X \to Y$ with $X,Y \in \XX_U$. Since $X \in \XX_U \subset \Fac U$, we may assume  $X \in \add U$ to show $\im \varphi \in \XX_U$.
  Since $Y$ belongs to $\XX_U$, there is a short exact sequence $0 \to U_1 \to U_0 \to Y \to 0$ with $U_1,U_0\in\add U$ by (1). By taking pullback, we obtain the following exact commutative diagram.
  \[
  \begin{tikzcd}
    0 \rar & U_1 \rar \dar[equal] & E \rar \ar[rd, phantom, "{\rm p.b.}"] \dar[twoheadrightarrow] & X \rar \dar[twoheadrightarrow]& 0 \\
    0 \rar & U_1 \dar[equal] \rar & Z \dar[hookrightarrow]\rar \ar[rd, phantom, "{\rm p.b.}"] & \im \varphi \rar\dar[hookrightarrow] & 0 \\
    0 \rar & U_1 \rar & U_0 \rar & Y \rar & 0
  \end{tikzcd}
  \]
  Now we have $\Ext_\Lambda^1(X,U_1) = 0$ by $X \in \add U$, which implies that the top short exact sequence splits. Thus $E \iso U_1 \oplus X \in \add U$ and $Z \in \Fac U \cap \Sub U$ hold. Therefore we get $Z \in \add U$ by Lemma \ref{ice:lem:image}. Then the middle horizontal short exact sequence implies $\im \varphi \in \XX_U$.

  (3)
  Obvious from the short exact sequence in (1)(a).
\end{proof}

Now we are ready to prove Theorem \ref{ice:thm:main}.
\begin{proof}[Proof of Theorem \ref{ice:thm:main}]
  Proposition \ref{ice:prop:icerigid} gives a map $P(-) \colon \icep \Lambda \to \rigid\Lambda$, and Proposition \ref{ice:prop:rigidice} gives a map $\ccok \colon \rigid \Lambda \to \icep \Lambda$.
  These propositions also show that these maps are mutually inverse to each other.

  Finally we prove statements for the representation-finite case. Let $\Lambda$ be a representation-finite hereditary artin algebra, that is, $\mod\Lambda$ has finitely many indecomposables up to isomorphism.
  Then \cite[Corollary 3.15]{enostr} implies that every subcategory of $\mod\Lambda$ closed under extensions has enough $\Ext$-projectives.
  In particular, we have $\ice\Lambda = \icep\Lambda$. Moreover, if $\CC$ is a CE-closed subcategory of $\mod\Lambda$, then Proposition \ref{ice:prop:icerigid} implies that $\CC = \ccok P(\CC)$ holds. Since $P(\CC)$ is rigid, $\CC$ is automatically closed under images by Proposition \ref{ice:prop:rigidice}.
\end{proof}

\section{Maps to torsion classes and wide subcategories}\label{ice:sec:4}
The class of ICE-closed subcategories contain both the classes of torsion classes and wide subcategories, so it is natural to ask the relation between these three classes. The aim of this section is to introduce two natural maps from the set of ICE-closed subcategories to the set of torsion classes and wide subcategories, and to investigate these maps via rigid modules.

Let us introduce some notation. For an artin algebra $\Lambda$, we denote by $\tors\Lambda$ (resp. $\ftors\Lambda$) the set of torsion classes (resp. functorially finite torsion classes) in $\mod\Lambda$.
Similarly, we denote by $\wide\Lambda$ (resp. $\fwide\Lambda$) the set of wide subcategories (resp. functorially finite wide subcategories) in $\mod\Lambda$.

First, we construct two maps $\TTT\colon\ice\Lambda \defl \tors\Lambda$ and $\WWW \colon \ice\Lambda \defl \wide\Lambda$ which are the identities on $\tors\Lambda$ and $\wide\Lambda$ respectively.
\begin{definition}
  Let $\Lambda$ be an artin algebra and $\CC$ an ICE-subcategory of $\mod\Lambda$.
  \begin{enumerate}
    \item $\TTT(\CC)$ denotes the smallest torsion class containing $\CC$.
    \item $\WWW(\CC)$ is a subcategory of $\CC$ defined as follows:
    \[
    \WWW(\CC) = \{ W \in \CC \, | \, \text{$\ker \varphi \in \CC$ for any map $\varphi \colon C \to W$ with $C \in \CC$} \}
    \]
  \end{enumerate}
\end{definition}
Clearly $\TTT(\TT) = \TT$ for $\TT \in \tors\Lambda$ and $\WWW(\WW) = \WW$ for $\WW \in \wide\Lambda$. It is non-trivial that $\WWW$ actually defines a map $\WWW \colon \ice\Lambda \to \wide\Lambda$, as we shall see below.
\begin{proposition}
  Let $\Lambda$ be an artin algebra and $\CC$ an ICE-closed subcategory of $\mod\Lambda$. Then $\WWW(\CC)$ is a wide subcategory of $\mod\Lambda$.
\end{proposition}
\begin{proof}
  This a special case of \cite[Theorem 4.5]{enomono}, but here we will give a proof using the general result in \cite[Exercise 8.23]{KS}.
  According to it, we say that an object $X \in \mod\Lambda$ is \emph{$\CC$-coherent} if $X \in \Fac\CC$ and $\ker \varphi \in \CC$ for every map $C \to X$ with $C\in\CC$.
  Since $\CC$ is closed under cokernels, it is easy to check that every $\CC$-coherent object belongs to $\CC$, namely, $\WWW(\CC)$ coincides with the category of $\CC$-coherent objects.
  Then since $\CC$ is extension-closed, \cite[Exercise 8.23]{KS} implies that $\WWW(\CC)$ is a wide subcategory of $\mod\Lambda$.
\end{proof}

If $\Lambda$ is hereditary, then $\TTT$ is equal to $\Fac$, as the following general proposition shows.
\begin{proposition}[{c.f. \cite[Proposition 2.13]{IT}}]
  Let $\Lambda$ be a hereditary artin algebra and $\CC$ an extension-closed subcategory of $\mod\Lambda$. Then $\Fac \CC = \TTT(\CC)$ holds, namely, $\Fac \CC$ is a torsion class.
\end{proposition}
\begin{proof}
  We only have to show that $\Fac \CC$ is closed under extensions. Take a short exact sequence
  \[
  \begin{tikzcd}
    0 \rar & L \rar & M \rar & N \rar & 0
  \end{tikzcd}
  \]
  with $L,N \in \Fac \CC$, and take surjections $\pi_L \colon C_L \defl L$ and $\pi_N \colon C_N \defl N$. Since $\Lambda$ is hereditary, the induced map $\Ext_\Lambda^1(N,C_L) \to \Ext_\Lambda^1(N,L)$ is a surjection. Thus we have the following exact commutative diagram, where we in addition take pullback along $\pi_N$.
  \[
  \begin{tikzcd}
    0 \rar & C_L \dar[equal]\rar & F\dar[twoheadrightarrow] \ar[rd, phantom, "{\rm p.b.}"] \rar & C_N\dar[twoheadrightarrow, "\pi_N"]\rar & 0 \\
    0 \rar & C_L \ar[rd, phantom, "{\rm p.o.}"] \dar["\pi_L"', twoheadrightarrow]\rar & E\dar[twoheadrightarrow] \rar & N \dar[equal]\rar & 0 \\
    0 \rar & L \rar & M \rar & N \rar & 0
  \end{tikzcd}
  \]
  Since $\CC$ is closed under extensions, we have $F \in \CC$. Thus we obtain $M \in \Fac \CC$.
\end{proof}

Next we will consider counterparts of the maps $\Fac$ and $\WWW$ in terms of rigid modules.
Let $\Lambda$ be a hereditary artin algebra. We denote by $\stilt \Lambda$ the set of isomorphism classes of basic support tilting $\Lambda$-modules. If $U$ a rigid $\Lambda$-module, then $\Fac U$ is closed under extension by \cite[Proposition 5.5, Corollary 5.9]{AS}, thus it is a torsion class. By the result of \cite{IT} or \cite{AIR}, there is a unique basic support tilting module $\ov{U}$ satisfying $\Fac \ov{U} = \Fac U$. We call $\ov{U}$ the \emph{co-Bongartz completion} of $U$.

Now the following proposition can immediately follows from definition, so we omit the proof.
\begin{proposition}\label{ice:prop:tiltcompati}
  Let $\Lambda$ be a hereditary artin algebra. Then the following diagram commutes, and the dashed maps are given by taking the co-Bongartz completion.
  \[
  \begin{tikzcd}
    \stilt \Lambda  \rar[shift left = .3ex, "\Fac"] \ar[dd, bend right=75, "1"'] \dar[hookrightarrow]
    & \ftors\Lambda \lar[shift left = .3ex, "P"] \ar[dd, bend left=75, "1"] \dar[hookrightarrow] \\
    \rigid \Lambda \rar[shift left = .3ex, "\ccok"] \dar[dashed, "(\ov{-})"'] \ar[rd, "\Fac", sloped]
    & \icep\Lambda  \dar["\Fac", twoheadrightarrow] \lar[shift left = .3ex, "P"] \\
    \stilt \Lambda  \rar[shift left = .3ex, "\Fac"] & \ftors\Lambda \lar[shift left = .3ex, "P"]
  \end{tikzcd}
  \]
\end{proposition}

Next we will investigate the map $\WWW \colon \icep\Lambda \to \wide\Lambda$, which is more non-trivial than $\TTT$. To do this, we introduce the \emph{$\Fac$-minimality}, \emph{covers} and \emph{split projectives}, following \cite{AS}.
\begin{definition}
  Let $\PP$ be a subcategory of $\mod\Lambda$. We say that $\PP$ is \emph{$\Fac$-minimal} if there is no proper subcategory $\PP'$ of $\PP$ satisfying $\PP \subset  \Fac \PP'$. We say that $U \in \mod\Lambda$ is \emph{$\Fac$-minimal} if $\add U$ is $\Fac$-minimal.
\end{definition}
Note that \emph{subcategories} are required to be closed under direct sums and direct summands.
If $U \in \mod\Lambda$ is basic and $U = \bigoplus_{i \in I} U_i$ with each $U_i$ indecomposable, then $U$ is $\Fac$-minimal if and only if there is no proper subset $J$ of $I$ satisfying $U \in \Fac (\bigoplus_{j \in J} U_j)$. $\Fac$-minimal basic modules are called \emph{covering-indecomposable modules} in \cite{AS}.
\begin{definition}
  Let $\CC$ be a subcategory of $\mod\Lambda$. Then an object $P$ in $\CC$ is \emph{split projective} if every surjection $C \defl P$ in $\mod\Lambda$ with $C \in \CC$ splits. We denote by $\PP_0(\CC)$ the subcategory of $\CC$ consisting of all split projective objects in $\CC$.
\end{definition}
It can be shown $\PP_0(\CC)$ is closed under direct sums and direct summands.
Clearly we have the inclusion $\PP_0(\CC) \subset \PP(\CC)$ for an extension-closed subcategory $\CC$ of $\mod\Lambda$.

Next we recall the notion of \emph{covers} of a category, introduced in \cite{AS}.
\begin{definition}
  Let $\CC$ be a subcategory of $\mod\Lambda$ and $\PP$ a subcategory of $\CC$.
  \begin{enumerate}
    \item $\PP$ is a \emph{cover of $\CC$} if $\CC \subset \Fac\PP$ holds.
    \item $\PP$ is a \emph{minimal cover of $\CC$} if $\PP$ is a cover of $\CC$ and there is no proper subcategory $\PP'$ of $\PP$ which is a cover of $\CC$.
    \item An object $P$ in $\CC$ is a (minimal) cover of $\CC$ if so is $\add P$.
    \item $\CC$ has a \emph{finite (minimal) cover} if there is an object $P$ in $\CC$ which is a (minimal) cover of $\CC$.
    \end{enumerate}
\end{definition}

We will use some results in \cite{AS} summarized as follows.
\begin{proposition}[{\cite[Theorem 2.3, Corollary 2.4]{AS}}]\label{ice:prop:facmin}
  Let $\Lambda$ be an artin algebra and $\CC$ a subcategory of $\mod\Lambda$. Then the following hold.
  \begin{enumerate}
    \item Let $\PP$ be a cover of $\CC$. Then $\PP$ is a minimal cover of $\CC$ if and only if $\PP$ is $\Fac$-minimal if and only if $\PP = \PP_0(\CC)$. In particular, the minimal cover of $\CC$ is unique if it exists.
    \item If $\CC$ has a finite cover, then $\CC$ has a finite minimal cover. Thus there is a unique basic $\Fac$-minimal cover $P$ of $\CC$ up to isomorphism, which satisfies $\PP_0(\CC) = \add P$.
  \end{enumerate}
\end{proposition}

By using this, we can define the following operation which yields a $\Fac$-minimal basic module.
\begin{definition}
  Let $\Lambda$ be an artin algebra and $M \in \mod\Lambda$. Then the \emph{$\Fac$-minimal version of $M$} is a basic $\Fac$-minimal cover $M_0$ of $\add M$, which is unique up to isomorphism by Proposition \ref{ice:prop:facmin}.
\end{definition}
If $M$ is a cover of a subcategory $\CC$ of $\mod\Lambda$, then Proposition \ref{ice:prop:facmin} implies that its $\Fac$-minimal version $M_0$ satisfies $\add M_0 = \PP_0(\CC)$.

Now we will use the following general observation when $\Ext$-projectives coincides with split projectives, which is of interest in its own.
\begin{proposition}\label{ice:prop:epikerfacmin}
  Let $\Lambda$ be an artin algebra and $\CC$ an extension-closed subcategory of $\mod\Lambda$ with enough $\Ext$-projectives. Then the following are equivalent:
  \begin{enumerate}
    \item $\PP_0(\CC) = \PP(\CC)$ holds, that is, every $\Ext$-projective object in $\CC$ is split projective in $\CC$.
    \item $\PP(\CC)$ is $\Fac$-minimal.
    \item $\CC$ is closed under epi-kernels, that is, for every short exact sequence
    \[
    \begin{tikzcd}
      0 \rar & L \rar & M \rar & N \rar & 0,
    \end{tikzcd}
    \]
    if $M$ and $N$ belong to $\CC$, then so does $L$.
  \end{enumerate}
\end{proposition}
\begin{proof}
  (1) $\Leftrightarrow$ (2): Immediate from Proposition \ref{ice:prop:facmin}, since $\PP(\CC)$ is a cover of $\CC$ by the enough $\Ext$-projectivity of $\CC$.

  (1) $\Rightarrow$ (3):
  Suppose that we have a short exact sequence $0 \to L \to M \to N \to 0$ with $M,N \in \CC$.
  Since $\CC$ has enough $\Ext$-projectives, there exists a short exact sequence $0 \to N' \to P \to N \to 0$ with $P \in \PP(\CC)$ and $N' \in \CC$. Then by taking the pullback, we obtain the following exact commutative diagram.
  \[
  \begin{tikzcd}
    & & 0 \dar & 0 \dar \\
    & & N' \rar[equal]\dar & N' \dar \\
    0 \rar & L \rar\dar[equal] & E \rar["\varphi"] \dar & P \dar \rar & 0 \\
    0 \rar & L \rar & M \dar \rar & N \dar \rar & 0 \\
    & & 0 & 0
  \end{tikzcd}
  \]
  Since $\CC$ is extension-closed, the middle vertical exact sequence implies $E \in \CC$. Since $P \in \PP(\CC)$, it is split projective by (1). Thus $\varphi$ splits, hence $L$ is a direct summand of $E$. This implies $L \in \CC$ since a subcategory $\CC$ is assumed to be closed under direct summands.

  (3) $\Rightarrow$ (1):
  Let $P$ be an $\Ext$-projective object in $\CC$, and take any surjection $\pi \colon C \defl P$ with $C \in \CC$. Then we have a short exact sequence $0 \to \ker \pi \to C \to P \to 0$, thus $\ker\pi \in \CC$ holds by (3). Therefore, since $P$ is $\Ext$-projective, this short exact sequence must split. This shows that $P$ is split projective.
\end{proof}

Now next we consider functorially finiteness of wide subcategories, which is of interest in its own. In particular, we will show that the functorially finiteness is equivalent to the exicestence of finite (co)cover, and to the contravariantly (covariantly) finiteness.

The following describes the relation between finite covers and covariantly finiteness.
\begin{lemma}[{\cite[Theorem 4.5, Proposition 3.7]{AS}}]\label{ice:lem:covfincover}
  Let $\Lambda$ be an artin algebra and $\CC$ a subcategory of $\mod\Lambda$ closed under images. Then the following are equivalent.
  \begin{enumerate}
    \item $\CC$ is covariantly finite.
    \item $\CC$ has a finite cover.
  \end{enumerate}
  In this case, let $\Lambda \to C^\Lambda$ be a left minimal $\CC$-approximation. Then $\PP_0(\CC) = \add C^\Lambda$ holds.
\end{lemma}

By using this, we obtain the following characterization of functorially finite wide subcategories.
\begin{proposition}\label{prop:ffwide}
  Let $\Lambda$ be an artin algebra and $\WW$ a wide subcategory of $\mod\Lambda$. Then the following are equivalent.
  \begin{enumerate}[font=\upshape]
    \item[(1)] $\WW$ is functorially finite.
    \item[(2)] $\WW$ is covariantly finite.
    \item[(2)$^\prime$] $\WW$ is contravariantly finite.
    \item[(3)] $\WW$ has an $\Ext$-progenerator.
    \item[(4)] $\WW$ is equivalent to $\mod\Gamma$ for some artin algebra $\Gamma$.
  \end{enumerate}
\end{proposition}
\begin{proof}
  Note that $\WW$ is closed under images, so we can apply Lemma \ref{ice:lem:covfincover}.

  (1) $\Rightarrow$ (2), (2)$^\prime$: Trivial.

  (2) $\Rightarrow$ (3):
  Take a left minimal $\WW$-approximation $\Lambda \to P$ with $P$ in $\WW$.
  Then $P$ is a minimal cover of $\Lambda$ with $\add P = \PP_0(\WW)$ by Lemma \ref{ice:lem:covfincover}.
  We claim that $P$ is an $\Ext$-progenerator of $\CC$. First, $P$ is $\Ext$-projective in $\WW$ since it is split projective.
  Second, since $P$ covers $\WW$ and $\WW$ is closed under kernels, for every object $W$ in $\WW$, there is a short exact sequence $0 \to W' \to P' \to W \to 0$ with $P' \in \add P$ and $W' \in \WW$. This shows that $\WW$ has enough $\Ext$-projectives with $\PP(\CC) = \add P = \PP_0(\CC)$.

  (3) $\Rightarrow$ (4):
  Let $P$ be an $\Ext$-progenerator of $\WW$. Define $\Gamma:= \End_\Lambda(P)$ and consider the functor $\Hom_\Lambda(P,-) \colon \mod\Lambda \to \mod \Gamma$. Then it easy to see that this induces an equivalence $\WW \equi \mod\Gamma$ by the standard argument in the Morita theory.

  (4) $\Rightarrow$ (2): Let $F \colon \mod\Gamma \equi \WW$ be an equivalence. Then the composition $\mod\Gamma \xrightarrow{F} \WW \hookrightarrow \mod\Lambda$ is exact since $\WW$ is closed under kernels and cokernels in $\mod\Lambda$. Thus $\add(F \Gamma)$ is a cover of $\WW$, since $\Gamma$ is a cover of $\mod\Gamma$ and $F$ preserves surjectivity.
  This implies that $\WW$ is covariantly finite by Lemma \ref{ice:lem:covfincover}.

  By the dual argument, (4) implies (2)$^\prime$. Thus (4) implies (1).
\end{proof}

Now let us return to the hereditary setting. The key observation is the following.

\begin{lemma}\label{ice:lem:p0sub}
  Let $\Lambda$ be a hereditary artin algebra and $\CC$ an ICE-closed subcategory of $\mod\Lambda$ with enough $\Ext$-projectives. Let $P$ be an object in $\PP_0(\CC)$ and $Q$ a submodule of $P$ satisfying $Q \in \CC$. Then $Q$ is also in $\PP_0(\CC)$.
\end{lemma}
\begin{proof}
  Since there is an $\Ext$-progenerator of $\CC$ by Proposition \ref{ice:prop:icerigid}, the category $\CC$ has a finite cover, thus has a minimal cover $\PP_0(\CC)$ by Proposition \ref{ice:prop:facmin}.
  Therefore there is a surjection $\pi \colon P_0 \defl Q$ with $P_0 \in \PP_0(\CC)$, hence we obtain a short exact sequence $0 \to \ker\pi \to P_0 \xrightarrow{\pi} Q \to 0$.

  Now since $\Lambda$ is hereditary, the map $\Ext_\Lambda^1(P,\ker \pi) \to \Ext_\Lambda^1(Q,\ker\pi)$ induced by the inclusion $Q \hookrightarrow P$ is surjective. Thus we obtain the following exact commutative diagram.
  \[
  \begin{tikzcd}
    & & 0 \dar & 0 \dar \\
    0 \rar & \ker\pi \rar\dar[equal] & P_0 \ar[rd, phantom, "{\rm p.b.}"] \rar["\pi"] \dar & Q \rar \dar & 0 \\
    0 \rar & \ker\pi \rar & E \rar["p"] \dar & P \rar \dar & 0 \\
    & & P/Q \rar[equal] \dar & P/Q \dar \\
    & & 0 & 0
  \end{tikzcd}
  \]
  Since $\CC$ is closed under cokernels, we have $P/Q \in \CC$. Then the middle vertical exact sequence implies $E \in \CC$ since $\CC$ is extension-closed. Now $p$ should split since $P$ is split projective in $\CC$, thus the middle horizontal short exact sequence splits. It follows that so does the top horizontal sequence, hence $Q$ is a direct summand of $P_0$. Therefore $Q \in \PP_0(\CC)$ holds.
\end{proof}

By using this, we can show the following last result in this section on the relation between $\Fac$-minimal rigid modules and wide subcategories. We denote by $\rigidz \Lambda$ the set of isomorphism classes of basic rigid $\Lambda$-modules which are $\Fac$-minimal.
\begin{proposition}\label{ice:prop:widefacmin}
  Let $\Lambda$ be a hereditary artin algebra. Then the following diagram commutes and the horizontal maps are bijections, where the map $(-)_0$ is given by taking the $\Fac$-minimal version.
  \[
  \begin{tikzcd}
    \rigidz \Lambda  \rar[shift left = .3ex, "\ccok"] \ar[dd, bend right=75, "1"'] \dar[hookrightarrow]
    & \fwide\Lambda \lar[shift left = .3ex, "P"] \ar[dd, bend left=75, "1"] \dar[hookrightarrow] \\
    \rigid \Lambda \rar[shift left = .3ex, "\ccok"] \dar["(-)_0"', twoheadrightarrow]
    & \icep\Lambda  \dar["\WWW", twoheadrightarrow] \lar[shift left = .3ex, "P"] \\
    \rigidz \Lambda  \rar[shift left = .3ex, "\ccok"] & \fwide\Lambda \lar[shift left = .3ex, "P"]
  \end{tikzcd}
  \]
  In particular, functorially finite wide subcategories are in bijection with basic $\Fac$-minimal rigid modules.
\end{proposition}
\begin{proof}
  It suffices to show the following for $U \in \rigid\Lambda$ by Theorem \ref{ice:thm:main}.
  \begin{enumerate}
    \item $\ccok U$ is wide if and only if $U$ is $\Fac$-minimal.
    \item $U_0$ is an $\Ext$-progenerator of $\WWW(\ccok U)$, where $U_0$ is the $\Fac$-minimal version of $U$.
  \end{enumerate}
  By Theorem \ref{ice:thm:main}, the category $\ccok U$ is an ICE-closed subcategory of $\mod\Lambda$ with $\PP(\ccok U) = \add U$.

  (1)
It is easy to check that an ICE-subcategory of $\mod\Lambda$ is wide if and only if it is closed under epi-kernels. Thus the assertion follows from Proposition \ref{ice:prop:epikerfacmin}.

  (2)
  Put $\CC:= \ccok U$. Then $\PP_0(\CC) = \add U_0$ holds by Proposition \ref{ice:prop:facmin}.
  First we show $\PP_0(\CC) \subset \WWW(\CC)$. It suffices to show that every map $\varphi\colon C \to P_0$ with $C \in \CC$ and $P_0 \in \PP_0(\CC)$ satisfies $\ker \varphi \in \CC$.
  Since $\CC$ is closed under images, we have $\im \varphi \in \CC$, and then Lemma \ref{ice:lem:p0sub} shows that $\im \varphi$ is split projective in $\CC$.
  Thus the induced surjection $C \defl \im\varphi$ splits, hence $\ker \varphi$ is a direct summand of $C$. Thus $\ker\varphi \in \CC$ holds.

  Now we have shown $U_0 \in \WWW(\CC)$. Moreover, $U_0$ is a cover of $\WWW(\CC)$ since it is a cover of $\CC$, and $U_0$ is split projective in $\WWW(\CC)$ since it is so in $\CC$.
  Then it is easy to see that $U_0$ is an $\Ext$-progenerator of $\WWW(\CC)$ because $\WWW(\CC)$ is closed under kernels.
\end{proof}

It was shown in \cite[Corollary 2.17]{IT} that the maps $\WWW \colon \ftors\Lambda \rightleftarrows \fwide\Lambda\colon \Fac$ are mutually inverse bijections (although their definition of $\fwide\Lambda$ is a bit different from ours). For the convenience of the reader, we give a short proof of this in our context.
\begin{corollary}
  Let $\Lambda$ be a hereditary artin algebra. Then $\WWW \colon \ftors\Lambda \rightleftarrows \fwide\Lambda\colon \Fac$ are mutually inverse bijections between the sets of functorially finite torsion classes and functorially finite wide subcategories.
\end{corollary}
\begin{proof}
  Let $U$ be a basic rigid $\Lambda$-module. According to Propositions \ref{ice:prop:tiltcompati} and \ref{ice:prop:widefacmin}, it suffices to show the following claims.
  \begin{enumerate}
    \item If $U$ is support tilting, then $U = P(\Fac U_0)$ holds, where $U_0$ is a $\Fac$-minimal version of $U$.
    \item If $U$ is $\Fac$-minimal, then $U$ is a $\Fac$-minimal version of $P(\Fac U)$.
  \end{enumerate}

  (1)
  By the definition of the $\Fac$-minimal version, we have $\Fac U_0 = \Fac U$. Thus the claim follows from $U = P(\Fac U)$, which holds by Theorem \ref{ice:thm:main}.

  (2)
  Since $U$ is $\Fac$-minimal, $U$ is a minimal cover of $\Fac U$. Since $P(\Fac U)$ is a cover of $\Fac U$, its $\Fac$-minimal version coincides with $U$.
\end{proof}

\section{Proof using exceptional sequences}\label{sec:5}
The most non-trivial part of our main theorem is that $\ccok P$ is ICE-closed for a rigid module $P$, Proposition \ref{ice:prop:rigidice} (2). In this section, we will give another proof of this using some results on \emph{exceptional sequences} over hereditary algebras.
As a by product, we will give the relation between ICE-closed subcategories and the tilting theory of wide subcategories.

The following is the key result in this section, whose proof requires exceptional sequences.
\begin{proposition}\label{prop:excmain}
  Let $\Lambda$ be a hereditary artin algebra and $P$ a rigid $\Lambda$-module. Denote by $\la P \ra_{\wide}$ the smallest wide subcategory containing $P$. Then $\la P \ra_{\wide}$ is equivalent to $\mod\Gamma$ for some hereditary artin algebra $\Lambda$, and $P$ is a tilting $\Gamma$-module under the equivalence $\la P \ra_{\wide} \equi \mod\Gamma$.
\end{proposition}
Consequently, we can reprove Proposition \ref{ice:prop:rigidice} (2) and give some new results. Observe that any torsion class in a wide subcategory of $\mod\Lambda$ (regarded as an abelian category) is an ICE-closed subcategory of $\mod\Lambda$, whose proof is straightforward and omitted.
\begin{corollary}\label{cor:exc}
  Let $\Lambda$ be a hereditary artin algebra. Then the following hold.
  \begin{enumerate}
    \item Let $P$ be a rigid $\Lambda$-module. Then $\ccok P$ is a torsion class in $\la P \ra_{\wide}$, thus is an ICE-closed subcategory of $\mod\Lambda$. Moreover, $\ccok P$ has an $\Ext$-progenerator $P$.
    \item We have bijections $\ccok \colon \rigid\Lambda\rightleftarrows \icep\Lambda \colon P(-)$.
    \item Every ICE-closed subcategory with enough $\Ext$-projectives is a functorially finite torsion class in some wide subcategory which is equivalent to $\mod\Gamma$ for some hereditary artin algebra $\Gamma$.
    \item If $\Lambda$ is representation-finite, then a subcategory $\CC$ of $\mod\Lambda$ is ICE-closed if and only if $\CC$ is a torsion class in some wide subcategory of $\mod\Lambda$.
  \end{enumerate}
\end{corollary}
\begin{proof}
  (1)
  Put $\WW:= \la P \ra_{\wide}$. By Proposition \ref{prop:excmain}, there is a hereditary artin algebra $\Gamma$ and an equivalence $F \colon \WW \equi \mod\Gamma$ such that $FP$ is a tilting $\Gamma$-module.
  Thus $\Fac FP$ is a torsion class in $\mod\Gamma$. We claim that $F$ induces an equivalence $\ccok P \equi \Fac FP$. Since $\WW$ is closed under cokernels, we have $\ccok P \subset \WW$. Since $F$ is an equivalence of abelian categories, it preserves cokernels, so $F(\ccok P) = \ccok FP$ holds.
  On the other hand, since $FP$ is a tilting $\Gamma$-module, it is well-known that $\Fac FP$ has an $\Ext$-progenerator $FP$ (see e.g. \cite[Theorem VI.2.5]{ASS}), which implies $\Fac FP = \ccok FP$. Thus $F(\ccok P) = \Fac FP$ holds, which gives an equivalence $\ccok P \equi \Fac FP$.

  Now since $\ccok P$ is a torsion class in a wide subcategory, $\ccok P$ is an ICE-closed subcategory of $\mod\Lambda$. Moreover, since it is exact equivalent to $\Fac FP$ via $F$, which has an $\Ext$-progenerator $FP$, it follows that $\ccok P$ has an $\Ext$-progenerator $P$.

  (2) and (3) follow from (1) and Proposition \ref{ice:prop:icerigid}.

  (4)
  The ``if" part follows from the remark before this corollary. Conversely, if $\Lambda$ is representation-finite, then every ICE-closed subcategory has enough $\Ext$-projectives by \cite[Corollary 3.15]{enostr}, thus (4) holds by (1).
\end{proof}
\begin{remark}
  In a forthcoming paper \cite{ES}, Corollary \ref{cor:exc} (4) will be generalized as follows: for any (possibly non-hereditary) length abelian category $\AA$, a subcategory $\CC$ of $\AA$ is ICE-closed if and only if $\CC$ is a torsion class in some wide subcategory of $\AA$.
\end{remark}

In the rest of this section, we prepare some results on exceptional sequences. \emph{In what follows, we assume that $\Lambda$ is a hereditary artin algebra}. For the basic facts on exceptional sequences on hereditary algebras, we refer the reader to \cite{CB,ringel,HK}.

First, we introduce the notion of \emph{perpendicular categories} in the sense of Geigel-Lenzing \cite{GL}, which is useful to study exceptional sequences.
\begin{definition}
  Let $\Lambda$ be an artin algebra and $\CC$ a collection of objects in $\mod\Lambda$. Then define two subcategories $\CC^\perp$ and $^\perp\CC$ of $\mod\Lambda$ as follows.
  \begin{enumerate}
    \item $\CC^\perp$ consists of modules $X$ with $\Hom_\Lambda(\CC,X) = 0 = \Ext_\Lambda^1(\CC,X)$.
    \item $^\perp\CC$ consists of modules $X$ with $\Hom_\Lambda(X,\CC) = 0 = \Ext_\Lambda^1(X,\CC)$.
  \end{enumerate}
\end{definition}
These subcategories are called right and left \emph{perpendicular categories} with respect to $\CC$.
\begin{definition}
  A sequence of $\Lambda$-modules $X = (X_1,X_2, \dots, X_r)$ is an \emph{exceptional sequence in $\mod\Lambda$} if it satisfies the following conditions.
  \begin{enumerate}
    \item Each $X_i$ is an indecomposable rigid $\Lambda$-module.
    \item $\Hom_\Lambda(X_j,X_i) = \Ext_\Lambda^1(X_j,X_i) = 0$ for any $i < j$.
  \end{enumerate}
\end{definition}
We will use the following basic results on exceptional sequences.
\begin{lemma}[{\cite[Corollary 4.2]{HR}}]\label{lem:exclem1}
  Let $P$ be a basic rigid $\Lambda$-module. Then there is an indecomposable decomposition $P = P_1 \oplus P_2 \cdots \oplus P_r$ such that $(P_1,P_2,\dots, P_r)$ is an exceptional sequence in $\mod\Lambda$.
\end{lemma}
\begin{lemma}[{\cite[Lemma 5]{CB}, \cite[Proposition 1]{ringel}, \cite[Theorem A.4]{HK}}]\label{lem:exclem2}
  Let $X = (X_1,\dots,X_r)$ be an exceptional sequence in $\mod\Lambda$. Then $\la X \ra_{\wide} = {}^\perp(X^\perp) = ({}^\perp X)^\perp$ holds, and $\la X \ra_{\wide}$ is equivalent to $\mod\Gamma$ such that $\Gamma$ is a hereditary artin algebra with $|\Gamma| = r$.
\end{lemma}
Now we can prove Proposition \ref{prop:excmain}.
\begin{proof}[Proof of Proposition \ref{prop:excmain}]
  Let $P$ be a rigid $\Lambda$-module. We may assume that $P$ is basic without loss of generality. Then Lemma \ref{lem:exclem1}, there is an exceptional sequence $(P_1, \dots, P_r)$ in $\mod\Lambda$ with $P = P_1 \oplus \cdots \oplus P_r$.

  Put $\WW := \la P \ra_{\wide}$. Then by Lemma \ref{lem:exclem2}, there are a hereditary artin algebra $\Gamma$ with $|\Gamma| = r$ and an equivalence $F \colon \WW \equi \mod\Gamma$.
  Consider $FP \in \mod \Gamma$. Since $\WW$ is extension-closed, we have $\Ext_\WW^1(P,P) = 0$, which implies $\Ext_\Gamma^1(FP,FP) = 0$. Thus $FP$ is partial tilting since $\Gamma$ is hereditary.
  Moreover, $|FP| = |P| = r = |\Gamma|$ holds, hence $FP$ must be a tilting $\Gamma$-module by the Bongartz completion argument, see e.g. \cite[Corollary VI.4.4]{ASS}.
\end{proof}

\begin{appendix}

\section{Enumerative results}
\emph{Throughout this appendix, we denote by $k$ a field}.
In this appendix, we give an explicit formula of the number $\# \rigid^i(kQ)$ of rigid $kQ$-modules with $i$ non-isomorphic direct summands for a Dynkin quiver $Q$.

In \cite[Proposition 6.1]{MRZ}, it was shown that this number does not depend on the orientation of $Q$, but the proof therein relies heavily on cluster combinatorics. In the first subsection, we give a short homological proof of this fact.
In the second subsection, we give an explicit formula of $\# \rigid^i(kQ)$ by using the enumerative result on cluster complexes in \cite{ftri}.

Let us introduce the set which we want to enumerate in this appendix.
\begin{definition}
  Let $\Lambda$ be an artin algebra. For a non-negative integer $i$, we denote by $\rigid^i \Lambda$ the set of isomorphism classes of basic $\Lambda$-modules $U$ satisfying $|U| = i$.
\end{definition}
Note that if $\Lambda$ is hereditary, then $\rigid^i \Lambda = \varnothing$ unless $0 \leq i \leq |\Lambda|$ by considering the Bongartz completion.

\subsection{Invariance under sink mutation}
Let us begin with recalling the \emph{mutation} of a quiver at a sink.
\begin{definition}
  Let $Q$ be a quiver. A \emph{sink} of $Q$ is a vertex $v$ such that there is no arrow starting at $v$. For a sink $v$ of $Q$, the \emph{sink mutation $\mu_v Q$} is a new quiver obtained from $Q$ by reversing all arrows which end at $v$.
\end{definition}
Note that $v$ becomes a source in $\mu_v Q$.
It is well-known that if two quivers $Q$ and $Q'$ have the same underlying graph which is a tree, then there is a sequence of sink mutations which transforms $Q$ into $Q'$.

Now the following is the main result in this subsection.
\begin{theorem}\label{ice:thm:mutinv}
  Let $Q$ be an acylic quiver with $n$ vertices and $v$ a sink of $Q$. Then for any $0 \leq i \leq n$, there is a bijection between $\rigid^i(kQ)$ and $\rigid^i(k (\mu_vQ))$.
\end{theorem}
This result immediately yields the following corollary, since any two Dynkin quivers with the same underlying graph can be connected via a series of sink mutations.
\begin{corollary}[{c.f. \cite[Proposition 6.1]{MRZ}}]\label{ice:cor:depend}
  Let $Q$ be a Dynkin quiver. Then $\#\rigid^i(kQ)$ only depends on the underlying Dynkin graph, not on the choice of an orientation.
\end{corollary}

In the rest of this subsection, we will give a proof of Theorem \ref{ice:thm:mutinv}.

\begin{definition}
  Let $Q$ be an acylic quiver and $v$ a sink or a source of $Q$.
  \begin{enumerate}
    \item We denote by $\mod_v kQ$ the subcategory consisting of $kQ$-modules which do \emph{not} contain $S(v)$ as a direct summand, where $S(v)$ is the simple module corresponding to $v$.
    \item For a $kQ$-module $M$, we denote by $M_v$ a unique module in $\mod_v kQ$ such that there is a following decomposition for some $n \geq 0$.
    \[
    M \iso M_v \oplus S(v)^{\oplus n}
    \]
    \item We define $\rigid^i_v (kQ):= \rigid^i (kQ) \cap \mod_v (kQ)$, that is, the set of basic rigid $kQ$-modules $M$ with $|M| = i$ such that $M$ \emph{does not} contain $S(v)$ as an direct summand.
    \item We define $\rigid^i_{\la v\ra} (kQ):= \rigid^i(kQ) \setminus \rigid^i_v(kQ)$, that is, the set of basic rigid $kQ$-modules $M$ with $|M| = i$ such that $M$ contains $S(v)$ as an direct summand.
  \end{enumerate}
\end{definition}
By definition, we have $\rigid^i(kQ) = \rigid_v^i(kQ) \sqcup \rigid_{\la v \ra}^i(kQ)$.
Our strategy is to construct two bijections $\rigid^i_v(kQ) \iso \rigid^i_v(k(\mu_v Q))$ and $\rigid^i_{\la v \ra} (kQ) \iso \rigid^i_{\la v \ra} (k (\mu_v Q))$ separately. The first one is established by the \emph{reflection functor}, and second one by the \emph{perpendicular categories}.

First, we will use the following property of the classical BGP-reflection functor. For the proof, we refer the reader to standard textbooks on quiver representation theory such as \cite{ASS}.
\begin{proposition}
  Let $Q$ be an acylic quiver and $v$ a sink of $Q$. There is a functor $R_v \colon \mod kQ \to \mod k(\mu_v Q)$ called the \emph{reflection functor}, which induces an equivalence $\mod_v kQ \equi \mod_v k (\mu_v Q)$. Moreover, this functor induces an isomorphism $\Ext_{kQ}^1(X,Y) \iso \Ext_{k(\mu_v Q)}^1(R_v X,R_v Y)$ for every $X,Y \in \mod_v kQ$.
\end{proposition}

This immediately yields the following bijection.
\begin{corollary}\label{ice:cor:1stbij}
  Let $Q$ be an acyclic quiver and $v$ a sink of $Q$. Then we have an bijection $R_v \colon \rigid_v^i(kQ) \xrightarrow{\sim} \rigid_v^i(k (\mu_v Q))$ given by $U \mapsto R_v U$.
\end{corollary}

Next we will construct a bijection $\rigid^i_{\la v \ra} (kQ) \iso \rigid^{i-1}(kQ_v)$, where $Q_v$ denotes a quiver obtained by removing $v$ from $Q$ by using perpendicular categories and their relation to (co)localizations. We remark that this lemma was proved in \cite[Lemma 3.2.5]{CK} in a more general setting.
\begin{lemma}\label{ice:lem:removing}
  Let $Q$ be an acylic quiver and $v$ a sink or a source of $Q$. Then both $S(v)^\perp$ and $^\perp S(v)$ are wide subcategories of $\mod kQ$, and equivalent to $\mod k Q_v$.
\end{lemma}
\begin{proof}
  The fact that these categories are wide subcategories follows from \cite[Proposition 1.1]{GL} since $kQ$ is hereditary, and its proof is quite straightforward, so we omit this.

  For the rest, we will use the theory of a \emph{recollement}. We refer the definitions and details to \cite{psa}. Let $e_v$ be an idempotent corresponding to $v$, and put $e = 1- e_v$. The we have the following recollement diagram, where $\mathsf{I}$ is the natural embedding functor.
  \[
  \begin{tikzcd}
    \mod \frac{kQ}{\la e \ra} \rar[hookrightarrow, "\mathsf{I}"]
    & \mod kQ \rar["\mathsf{E}"] \lar[bend right = 40] \lar[bend left = 40]
    & \mod e(kQ)e \lar[bend left = 40, hookrightarrow, "\mathsf{R}"] \lar[bend right = 40,hookrightarrow, "\mathsf{L}"']
  \end{tikzcd}
  \]
  Here $\mathsf{E}:= \Hom_{kQ}(e(kQ),-) = (-)e$, and $\mathsf{L}$ and $\mathsf{R}$ are left and right adjoint functors of $\mathsf{E}$, and both are fully faithful.
  Consider the essential image $\im\mathsf{I}$ of $\mathsf{I}$.
  Since $\mathsf{E}$ is a localization and a colocaliztion with respect to the Serre subcategory $\im\mathsf{I}$ \cite[Remark 2.2]{psa}, we have that $\im\mathsf{R}$ and $\im\mathsf{L}$ coincide with the perpendicular category $(\im \mathsf{I})^\perp$ and $^\perp(\im\mathsf{I})$ respectively \cite[Proposition 2.2]{GL}.
  On the other hand, $\im\mathsf{I}$ consists of modules $M$ with $Me = 0$, thus $\im\mathsf{I} = \add S(v)$ holds.
  Therefore, we have $\im\mathsf{R} = S(v)^\perp$ and $\im\mathsf{L}={}^\perp S(v)$. Since $\mathsf{L}$ and $\mathsf{R}$ are fully faithful, both $S(v)^\perp$ and $^\perp S(v)$ are equivalent to $\mod e(kQ)e$. Now the assertion holds since we clearly have an isomorphism of algebras $e(kQ)e \iso k Q_v$.
\end{proof}

By using this, we obtain the following bijection.
\begin{proposition}\label{ice:prop:red}
  Let $Q$ be an acylic quiver and $v$ a sink or a source of $Q$. Then for $i \geq 1$, we have a bijection
  \[
  \rigid^i_{\la v \ra}(kQ) \iso \rigid^{i-1}(k Q_v).
  \]
  The map is given by $U \mapsto (U_v)e \in \mod e(kQ)e$, where $e$ is the same as in the proof of Lemma \ref{ice:lem:removing}, and we identify $e (kQ) e $ with $k Q_v$.
\end{proposition}
\begin{proof}
  We give a proof for the case $v$ is a sink, and the same proof applies for the source case by using $S(v)^\perp$ instead of $^\perp S(v)$. First recall that these subcategories are wide subcategories of $\mod kQ$ by Lemma \ref{ice:lem:removing}, thus are abelian categories.
  The key observation is the following claim.

  {\bf (Claim)}:
  \emph{For $X \in \mod_v(kQ)$, the following are equivalent:
  \begin{enumerate}
    \item $X \oplus S(v)$ is rigid.
    \item $X \in {}^\perp S(v)$ holds, and $X$ is rigid in the abelian category ${}^\perp S(v)$.
  \end{enumerate}
  }
  \emph{Proof of (Claim)}.

  (1) $\Rightarrow$ (2):
  Since $S(v)\oplus X$ is rigid, $\Ext_{kQ}^1(X,S(v)) = 0$ holds. Moreover, if we have a non-zero map $X \to S(v)$, then it must be surjective since $S(v)$ is simple, hence it splits since $S(v)$ is projective. This contradicts to $X \in \mod_v kQ$, therefore we have $\Hom_{kQ}(X,S(v)) = 0$.
  It follows that $X \in {}^\perp S(v)$ holds.
  Since ${}^\perp S(v)$ is a wide subcategory of $\mod kQ$, an $\Ext^1$ inside ${}^\perp S(v)$ is the same as an $\Ext^1$ inside $\mod kQ$. Thus $X$ is rigid in the abelian category ${}^\perp S(v)$ since so is in $\mod kQ$.

  (2) $\Rightarrow$ (1):
  By the above argument, $X$ is a rigid $kQ$-module, and $\Ext_{kQ}^1(S(v),X \oplus S(v))$ vanishes since $S(v)$ is projective. Thus $X\oplus S(v)$ is rigid by $X \in {}^\perp S(v)$. $\qedb$

  By (Claim), the map $U \mapsto U_v$ clearly induces the following bijection
  \[
  \rigid^i_{\la v \ra} (kQ) \iso \rigid^{i-1}({}^\perp S(v)),
  \]
  where the right hand side is the set of basic rigid objects $X$ in the abelian category ${}^\perp S(v)$ satisfying $|X| = i-1$. Now the assertion holds from Lemma \ref{ice:lem:removing}, since $^\perp S(v)$ is equivalent to $\mod kQ_v$ as abelian categories.
\end{proof}

Now we immediately obtain the second bijection.
\begin{corollary}\label{ice:cor:2ndbij}
  Let $Q$ be an acylic quiver and $i$ a sink of $Q$. Then there is a bijection between $\rigid_{\la v \ra}^i (kQ)$ and $\rigid_{\la v \ra}^i k(\mu_v Q)$.
\end{corollary}
\begin{proof}
  Note that $v$ is a source of $\mu_v Q$. Then by Proposition \ref{ice:prop:red}, we have two bijections between $\rigid_{\la v \ra}^i (kQ) \iso \rigid^{i-1}(k Q_v)$ and $\rigid_{\la v \ra}^i k(\mu_v Q) \iso \rigid^{i-1}k(\mu_v Q)_v$.
  Since $(\mu_v Q)_v = Q_v$ holds, we obtain a bijection between $\rigid_{\la v \ra}^i (kQ)$ and $\rigid_{\la v \ra}^i k(\mu_v Q)$ by composing the above two bijections.
\end{proof}

Now we are ready to prove Theorem \ref{ice:thm:mutinv}.
\begin{proof}[Proof of Theorem \ref{ice:thm:mutinv}]
  We have the following equalities by definition.
  \begin{align*}
    \rigid^i(kQ) &= \rigid_v^i(kQ) \sqcup \rigid_{\la v \ra}^i(kQ) \\
    \rigid^i(k(\mu_vQ)) &= \rigid_v^i(k(\mu_vQ)) \sqcup \rigid_{\la v \ra}^i(k(\mu_vQ))
  \end{align*}
   Now we have a bijection $\rigid_v^i(kQ) \iso  \rigid_v^i(k(\mu_vQ))$ by Corollary \ref{ice:cor:1stbij}, and a bijection $\rigid_{\la v \ra}^i(kQ) \iso \rigid_{\la v \ra}^i(k(\mu_vQ))$ by Corollary \ref{ice:cor:2ndbij}. Thus by combining these two, we obtain a bijection between $\rigid^i(kQ)$ and $\rigid^i(k(\mu_vQ))$.
\end{proof}

\subsection{Formula for the number of rigid modules}
In this subsection, we give an explicit formula for $\#\rigid^i (kQ)$ for a Dynkin quiver $Q$.
\emph{From now on, we assume that $Q$ is a Dynkin quiver of type $X_n \in \{ A_n, D_n, E_6,E_7,E_8\}$ with $n$ vertices.}

Let $\Delta (Q)$ be a simplicial complex defined as follows: the set of vertices is $\rigid (kQ)$, and an $(i-1)$-simplex consists of sets of rigid $kQ$-modules whose direct sum is rigid, or equivalently, belongs to $\rigid^i(kQ)$.
This complex was introduced by Riedtmann and Schofield \cite{RS}. By definition, $\#\rigid^i(kQ)$ is equal to the number of $(i+1)$-faces of $\Delta(Q)$, thus the calculation of $\#\rigid^i(kQ)$ is nothing but that of the face vector of $\Delta(Q)$. Although this complex is classical, there seems to be no papers which contain an explicit formula of $\#\rigid^i(kQ)$.

We give such an formula, by translating our problem to a combinatorial problem on a \emph{cluster complex}.
Let $\Phi$ be the root system of type $X_n$, and let $\Phi_{\geq -1}$ denote the set of \emph{almost positive roots} of $\Phi$, that is, positive roots together with negative simple roots.
Then the \emph{cluster complex} $\Delta(X_n)$ of type $X_n$, also known as the \emph{generalized associahedron}, is a simplicial complex with the vertex set $\Phi_{\geq -1}$.
We refer the reader to \cite{FZ,MRZ} for the details. Then this complex contains $\Delta(Q)$ if $Q$ is \emph{bipirtite}, that is, every vertex is either a sink or a source. More precisely, the following holds.
\begin{proposition}\label{ice:prop:poscon}
  Let $Q$ be a Dynkin quiver with a bipirtite orientation. Then taking dimension vectors, we have an embedding $\Delta(Q) \hookrightarrow \Delta(X_n)$, which induces an isomorphism between $\Delta(Q)$ and the full subcomplex of $\Delta(X_n)$ spanned by positive roots.
\end{proposition}
We refer the reader to \cite[4.12]{MRZ} for the proof, and to \cite[Theorem 4.5]{BMRRT} for the more theoretical explanation of this using the cluster category.
As a corollary, we have the following equality.
\begin{corollary}
  Let $Q$ be a Dynkin quiver. Then $\#\rigid^i(kQ)$ is equal to the number of $(i-1)$-faces of $\Delta(X_n)$ which contain no negative simple roots.
\end{corollary}
\begin{proof}
  We can transform $Q$ into a bipirtite Dynkin quiver by using sink mutations. Thus we may assume that $Q$ is bipirtite by Corollary \ref{ice:cor:depend}. Then the assertion is immediate from Proposition \ref{ice:prop:poscon}.
\end{proof}

Now we are ready to show the formula of $\#\rigid^i(kQ)$ by using \cite{ftri}, which enumerates the number of faces of $\Delta(Q)$ satisfying various conditions.

\begin{theorem}\label{ice:thm:count}
  Let $Q$ be a Dynkin quiver of type $X_n$. Then the number $\#\rigid^i(kQ)$ is equal to {\upshape ($X_n$)} in the following list, where $\binom{n}{i}$ denotes the binomial coefficient.
  \begin{itemize}[font=\upshape, itemsep = 2ex]
    \item[($A_n$)] $\displaystyle \frac{1}{i+1} \binom{n}{i}\binom{n+i}{i}$
    \item[($D_n$)]
    $\displaystyle
    \binom{n}{i}\binom{n+i-2}{i} + \binom{n-1}{i-1}\binom{n+i-3}{i-1}
    - \frac{1}{n-1} \binom{n-1}{i-1}\binom{n+i-2}{i}
    $
    \item[($E_6$)]
      \begin{tabular}{|C||C|C|C|C|C|C|C||C|}
        \hline
        i & 0 & 1 & 2 & 3 & 4 & 5 & 6 & \text{\upshape total} \\ \hline
        & 1 & 36 & 300 & 1035 & 1720 & 1368 & 418 & 4878 \\\hline
      \end{tabular}

    \item[($E_7$)]
      \begin{tabular}{|C||C|C|C|C|C|C|C|C||C|}
        \hline
        i & 0 & 1 & 2 & 3 & 4 & 5 & 6 & 7 & \text{\upshape total}  \\ \hline
        & 1 & 63 & 777 & 3927 & 9933 & 13299 & 9009 & 2431 & 39440 \\ \hline
      \end{tabular}

    \item[($E_8$)]
      \begin{tabular}{|C||C|C|C|C|C|C|C|C|C||C|}
        \hline
        i & 0 & 1 & 2 & 3 & 4 & 5 & 6 & 7 & 8 & \text{\upshape total} \\ \hline
        & 1 & 120 & 2135 & 15120 & 54327 & 108360 & 121555 & 71760 & 17342 & 390720 \\ \hline
      \end{tabular}
  \end{itemize}
\end{theorem}
\begin{proof}
  The computation is achieved by specializing the results \cite[Theorems FA, FD, Section 7]{ftri} to $m=1$ and $y=0$. More precisely, in \cite{ftri}, the \emph{$m$-generalization} of cluster complexes are studied, and $m=1$ is the classical case. Then the number of faces of an $m$-cluster complex which consists of given numbers of positive roots and negative roots was computed, and $y=0$ means that we exclude negative roots.
\end{proof}

By our main result, $\#\rigid (kQ)$ is equal to the number of ICE-closed subcategories of $\mod kQ$.
Since we have $\#\rigid(kQ) = \sum_{i=0}^{n}\#\rigid^i(kQ)$, we obtain the following enumeration.
\begin{corollary}\label{ice:cor:count}
  Let $Q$ be a Dynkin quiver of type $X_n$. Then the number of ICE-closed subcategories in $\mod kQ$ is equal to the sum of the numbers given in Theorem \ref{ice:thm:count} over $i=0,1,\dots,n$. In particular, if $Q$ is of type $A_n$, then the equality holds,
  \[
    \#\ice (kQ) = \sum_{i=0}^n \frac{1}{i+1}\binom{n}{i}\binom{n+i}{i},
  \]
  where the right hand side is known as the $n$-th large Schr\"oder number \cite[A006318]{OEIS}.
\end{corollary}

\begin{remark}
  In \cite[Theorem 6.13]{enomono}, the author computes $\#\ice (kQ)$ for a linearly oriented $A_n$ quiver by a different method. By combining this with Theorem \ref{ice:thm:mutinv}, we can give another proof of the fact that $\#\ice kQ$ is equal to the $n$-th large Schr\"oder number for a quiver of type $A_n$.
\end{remark}

\end{appendix}

\begin{ack}
  The author would like to thank Osamu Iyama for helpful comments. He would also like to thank Arashi Sakai and Yasuaki Gyoda for helpful discussions.
  This work is supported by JSPS KAKENHI Grant Number JP18J21556.
\end{ack}

\end{document}